\documentclass{amsart}%
\usepackage{mathrsfs}
\usepackage{amsmath}
\usepackage{amsfonts}
\usepackage{amssymb}
\usepackage{graphicx}
\usepackage{amsthm}
\usepackage{hyperref}
\usepackage{latexsym}%
\providecommand{\U}[1]{\protect\rule{.1in}{.1in}}
\newtheorem{theorem}{Theorem}[section]

\newtheorem{corollary}[theorem]{Corollary}

\newtheorem{lemma}[theorem]{Lemma}
\newtheorem{proposition}[theorem]{Proposition}

\theoremstyle{definition}

\newtheorem{fact}[theorem]{Fact}

\begin{document}
\title[F.A. of Transitive Logics of Finite Depth and Finite Weak Width]{Finite Axiomatizability  of Transitive Logics of Finite Depth and of Finite Weak Width}
\author {Yan Zhang}
\address {Department of Philosophy, Renmin University of China}

\begin{abstract}
This paper presents a study of the finite axiomatizability of transitive
logics of finite depth and finite weak width. We prove the finite axiomatizability
of each transitive logic of finite depth and of weak width $1$ that
is characterized by rooted transitive frames in which all antichains contain at most $n$ irreflexive points. As a negative result,
we show that there are non-finitely-axiomatizble transitive logics
of depth $n$ and of weak width $k$ for each $n\geqslant3$ and $k\geqslant2$.

\end{abstract}

\keywords{Modal logic, finite axiomatizability, transitive logics, finite depth, finite width}
\maketitle

\section*{}

This paper presents a study of finite axiomatizability of transitive
logics of finite depth and of finite weak width, i.e., containing a
weak width formula $\mathsf{Wid}_{n}^{+}$ for some $n\geqslant1$.
These formulas are weaker forms of width formulas in \cite{Fine-74-logics-containing-K4-part-1},
and each $\mathsf{Wid}_{n}^{+}$ ($n\geqslant1$) corresponds to the
condition within  rooted transitive frames that all subframes generated
by some proper successor of a root are at most width $n$. As a negative
results, we show that there are non-finitely-axiomatizble transitive
logics of depth $n$ and of weak width $k$ for each $n\geqslant3$
and $k\geqslant2$, by a way of constructing infinite irreducible
sequences of frames. As a positive result, we prove the finite axiomatizability
of each transitive logic of finite depth and of weak width $1$ that
contains $\mathsf{Wid}_{n}^{\bullet}$ for an $n\geqslant1$, in which
$\mathsf{Wid}_{n}^{\bullet}$ corresponds to the condition within
rooted transitive frames that each antichain in them contains at most
$n$ irreflexive points.

Section \ref{sect: preliminaries} provides preliminary notions and
facts, and section \ref{sect: seq of frames and seq of lists} gives
criteria of finite axiomatizability of transitive logics whose extensions
all have the f.m.p. In section \ref{sec:Transitive-Logics-of FD and FSW},
we introduce transitive logics of finite depth and of finite weak width,
and prove the non-finite-axiomatizability result. We then present
in section \ref{sect: FA of lgcs of finite suc-eq-width} our main
result of finite axiomatizability of transitive logic of finite depth
and of weak width $1$, and concludes the paper in section \ref{sect: remarks}.

\section{Preliminaries\label{sect: preliminaries}}

This section lists some standard preliminary notions and theorems,
and a full account can be found in standard modal logic textbooks
(e.g., \cite{Chagrov-Zakharyaschev-ml-bk-97}, \cite{Blackburn-de-Rijke-ml-book}).
\emph{Modal formulas} are built up from propositional variables, using
truth-functional operators and the necessity operator $\Box$. We
will simply call them formulas. A \emph{normal modal logic} (or simply
modal logic) is a set of modal formulas that contains all truth-functional
tautologies and $\Box(p\rightarrow q)\rightarrow(\Box p\rightarrow\Box q)$,
and is closed under modus ponens, substitution and necessitation.
As usual, we use $\mathbf{K}$ ($\mathbf{K4}$) for the smallest
modal logic (containing $\Box p\rightarrow\Box\Box p$). For each
modal logic $\mathbf{L}$, an \emph{extension} of $\mathbf{L}$ is
a modal logic $\mathbf{L}^{\prime}$ such that $\mathbf{L}\subseteq\mathbf{L}^{\prime}$.
Let $\mathbf{L}$ be any modal logic and let $\Delta$ be any set
of formulas, $\mathbf{L}\oplus\Delta$ is the smallest modal logic
including $\mathbf{L}\cup\Delta$; and for each formula $\phi$, we
use $\mathbf{L}\oplus\phi$ for $\mathbf{L}\oplus\{\phi\}$. As usual,
we use $\mathbf{S4}$ for $\mathbf{K4}\oplus\Box p\rightarrow p$.
A modal logic $\mathbf{L}^{\prime}$ is \emph{finitely axiomatizable
over} $\mathbf{L}$ if $\mathbf{L}^{\prime}=\mathbf{L}\oplus\Delta$
for a finite $\Delta$, and is \emph{finitely axiomatizable} if it
is finitely axiomatizable over $\mathbf{K}$.

Let $\mathfrak{F}=\left\langle W,R\right\rangle $ be any frame with
$w\in W$, and let $\mathfrak{M}$ be any model on $\mathfrak{F}$.
For each formula $\phi$, we use $\mathfrak{M},w\vDash\phi$ for that
$\mathfrak{M}$ satisfies $\phi$ at $w$, use $\mathfrak{F},w\vDash\phi$
for that $\mathfrak{M}^{\prime},w\vDash\phi$ for each model $\mathfrak{M}^{\prime}$
on $\mathfrak{F}$, and use $\mathfrak{F}\vDash\phi$ for that $\phi$
is valid in $\mathfrak{F}$ ($\mathfrak{F}$ is a frame for $\phi$).
For any set $\Delta$ of formulas, and any class $\mathscr{C}$ of
frames, the validity-relation $\vDash$ between them are defined as
usual, and we will use $\mathbf{Log}(\mathscr{C})$ for the modal
logic $\{\phi:\mathscr{C}\vDash\phi\}$. For all $u,v\in W$, let
$\vec{R}uv$ iff $Ruv$ but not $Rvu$, and let $u\perp_{R}v$ iff
neither $Ruv$ nor $Rvu$. For all $u,v\in W$, when $Ruv$, we say
that $u$ \emph{sees} $v$, and call $v$ a \emph{successor of} $u$;
and when $\vec{R}uv$, we call $v$ a \emph{proper successor} of $u$,
and $u$ a \emph{proper predecessor} of $v$. For each $X\subseteq W$,
let $\left.X\right\uparrow _{R}=\{v:Ruv$ for a $u\in X\}$, $\left.X\right\downarrow _{R}=\{v:Rvu$
for a $u\in X\}$, $\left.X\right\uparrow _{R}^{-}=\left.X\right\uparrow _{R}-X$
and $\left.X\right\downarrow _{R}^{-}=\left.X\right\downarrow _{R}-X$.
When $R$ is clear in the context, we drop ``$_{R}$”\ and use ``$\left.X\right\uparrow $”\ and
``$\left.X\right\downarrow $”\ instead. For each $w\in W$, let
$\left.w\right\uparrow =\{w\}{\uparrow}$, $\left.w\right\downarrow =\{w\}{\downarrow}$,
$\left.w\right\uparrow ^{-}=\{w\}{\uparrow}^{-}$ and $\left.w\right\downarrow ^{-}=\{w\}{\downarrow}^{-}$.

For each family $\{\mathfrak{F}_{i}\}_{i\in I}$ ($\{\mathfrak{M}_{i}\}_{i\in I}$)
of pairwise disjoint frames (models), we use $\biguplus_{i\in I}\mathfrak{F}_{i}$
($\biguplus_{i\in I}\mathfrak{M}_{i}$) for the disjoint union of
$\{\mathfrak{F}_{i}\}_{i\in I}$ ($\{\mathfrak{M}_{i}\}_{i\in I}$).
For each frame $\mathfrak{F}=\left\langle W,R\right\rangle $ and
each model $\mathfrak{M}$ on $\mathfrak{F}$, and for each nonempty
$X\subseteq W$, we use $\mathfrak{F}\upharpoonright X$ ($\mathfrak{M}\upharpoonright X$)
for the restriction of $\mathfrak{F}$ ($\mathfrak{M}$) to $X$,
and use $\mathfrak{F}|_{X}$ ($\mathfrak{M}|_{X}$) for the subframe
of $\mathfrak{F}$ (submodel of $\mathfrak{M}$) generated by $X$;
and when $X=\{w\}$, we use $\mathfrak{F}|_{w}$ and $\mathfrak{M}|_{w}$
for $\mathfrak{F}|_{\{w\}}$ and $\mathfrak{M}|_{\{w\}}$ respectively.
For frames $\mathfrak{F}$ and $\mathfrak{G}$ (models $\mathfrak{M}$
and $\mathfrak{M}^{\prime}$), we say that a function $f$ \emph{reduces}
$\mathfrak{F}$ ($\mathfrak{M}$) to $\mathfrak{G}$ ($\mathfrak{M}^{\prime}$)
when $f$ is a reduction of $\mathfrak{F}$ ($\mathfrak{M}$) to $\mathfrak{G}$
($\mathfrak{M}^{\prime}$); and that $\mathfrak{F}$ ($\mathfrak{M}$)
is \emph{reducible to} $\mathfrak{G}$ ($\mathfrak{M}^{\prime}$)
if a function reduces $\mathfrak{F}$ ($\mathfrak{M}$) to $\mathfrak{G}$
($\mathfrak{M}^{\prime}$). We assume the reader's familiarity with
the related theorems on preservation of truth and validity under these
frame/model constructions. 

Let $\mathfrak{F}=\left\langle W,R\right\rangle $ be a transitive
frame. For all $w,u\in W$, $w\sim_{R}u$ iff either $w=u$, or $Rwu$
and $Ruw$. A \emph{cluster} in $\mathfrak{F}$ is an equivalence
class modulo $\sim_{R}$. For each $w\in W$, we use $\boldsymbol{c}_{(w)}$
for the cluster containing $w$. For each cluster $\boldsymbol{c}$
in $\mathfrak{F}$, $\boldsymbol{c}$ is \emph{degenerate} if it is
a singleton of an irreflexive point in $\mathfrak{F}$, otherwise,
it is \emph{nondegenerate.} Let $k\geqslant1$. A point $u_{1}$
in $\mathfrak{F}$ is \emph{of rank greater than} $k$ if there is
an $\vec{R}$-chain $\{u_{1},\ldots,u_{n}\}$ with $n>k$, and is
of \emph{rank} $k$ if there is an $\vec{R}$-chain $\{u_{1},\ldots,u_{k}\}$
and $u_{1}$ is not of rank greater than $k$. $\mathfrak{F}$ is
of \emph{rank} $k$ if it contains a point of rank $k$ but no point
of rank greater than $k$, and is \emph{of finite rank} if it is of
rank $k$ for some $k\geqslant1$. The following formulas are from
\cite{Segerberg-essay-MD}, where $i\geqslant1$:
\begin{align*}
\mathsf{B}_{1} & =\Diamond\Box p_{1}\rightarrow p_{1}\text{,}\\
\mathsf{B}_{i+1} & =\Diamond(\Box p_{i+1}\wedge\lnot\mathsf{B}_{i})\rightarrow p_{i+1}\text{.}
\end{align*}
We use $\mathbf{K4B}_{n}$ ($\mathbf{S4B}_{n}$ ) for $\mathbf{K4}\oplus\mathsf{B}_{n}$
($\mathbf{S4}\oplus\mathsf{B}_{n}$), where $n\geqslant1$. A transitive
logic is \emph{of depth} $n$ ($n\geqslant1$) if it contains $\mathsf{B}_{n}$
but not $\mathsf{B}_{n-1}$ (it is assumed that $\mathsf{B}_{0}=\bot$),
and is \emph{of finite depth} if it contains $\mathsf{B}_{k}$ for
a $k\geqslant1$. The following are established in \cite{Segerberg-essay-MD}:

\begin{proposition} For each transitive frame $\mathfrak{F}$ and
each $n\geqslant1$, $\mathfrak{F}\vDash\mathsf{B}_{n}$ iff $\mathfrak{F}$
is of rank at most $n$.\label{prop: frame cond 4 B_k} \end{proposition}

\begin{theorem} \label{thm: Segerberg, depth}All transitive logics
of finite depth have the f.m.p. \end{theorem}

An \emph{antichain} in a transitive frame $\mathfrak{F}=\left\langle W,R\right\rangle $
is a set  $A\subseteq W$ such that for all $u,v\in A$, $u\neq v$
only if $u\perp_{R}v$. Whenever we speak of an antichain $\{u_{0},\ldots,u_{n}\}$
in a frame, we presuppose that $u_{0},\ldots,u_{n}$ are distinct.
A transitive frame is \emph{of width at most} $n$ ($n\geqslant1$)
if $\left\vert A\right\vert \leqslant n$ for each antichain $A$
in the frame. The following formulas are from \cite{Fine-74-logics-containing-K4-part-1},
where $n\geqslant1$:
\[
\mathsf{Wid}_{n}={
{\textstyle \bigwedge\nolimits _{i\leqslant n}}
}\Diamond p_{i}\rightarrow{
{\textstyle \bigvee\nolimits _{0\leqslant i\neq j\leqslant n}}
}\Diamond(p_{i}\wedge(p_{j}\vee\Diamond p_{j}))\text{.}
\]
A transitive logic is \emph{of width} $n$ ($n\geqslant1$) if it
contains $\mathsf{Wid}_{n}$ but not $\mathsf{Wid}_{n-1}$ (it is
assumed that $\mathsf{Wid}_{0}=\bot$), and is \emph{of finite width}
if it contains $\mathsf{Wid}_{k}$ for a $k\geqslant1$. The following
proposition is from \cite{Fine-74-logics-containing-K4-part-1}:

\begin{proposition} For each rooted transitive frame $\mathfrak{F}$
and each $n\geqslant1$, $\mathfrak{F}\vDash\mathsf{Wid}_{n}$ iff
$\mathfrak{F}$ is of width at most $n$.\label{prop: frame cond 4 Wid_n}
\end{proposition}

\section{Criteria of Finite Axiomatizability\label{sect: seq of frames and seq of lists}}

In this section, we present necessary and sufficent conditions for
all extensions of a modal logic $\mathbf{L}$ to be finitely axiomatizable
over $\mathbf{L}$.  For each family $\{\mathbf{L}_{i}\}_{i\in I}$
of modal logics, $\bigoplus\nolimits _{i\in I}\mathbf{L}_{i}$ is
the smallest modal logic including $\bigcup\nolimits _{i\in I}\mathbf{L}_{i}$.
The following is a well-known theorem from Tarski (see, e.g., \cite{Chagrov-Zakharyaschev-ml-bk-97}):

\begin{theorem} Let $\mathbf{L}$ and $\mathbf{L}^{\prime}$ be any
modal logics such that $\mathbf{L}\subseteq\mathbf{L}^{\prime}$.
Then $\mathbf{L}^{\prime}$ is finitely axiomatizable over $\mathbf{L}$
iff there is no infinite ascending $\subset$-chain $\mathbf{L}_{0}\subset\mathbf{L}_{1}\subset\mathbf{L}_{2}\subset\cdots$
of extensions of $\mathbf{L}$ such that $\mathbf{L}^{\prime}=
{\textstyle \bigoplus\nolimits _{i\in\omega}}
\mathbf{L}_{i}$.\label{thm: Tarski's criterion of f. ax.} \end{theorem}

Let $\{\mathfrak{F}_{i}\}_{i\in\omega}$ be any infinite sequence
of frames. $\{\mathfrak{F}_{i}\}_{i\in\omega}$ is \emph{backward
irreducible} (\emph{forward-backward irreducible}, or simply \emph{irreducible})
if for all $i,j\in I$ with $i<j$ ($i\neq j$), no point-generated
subframe of $\mathfrak{F}_{j}$ is reducible to $\mathfrak{F}_{i}$.
For each class $
\mathscr{C}
$ of frames, $\{\mathfrak{F}_{i}\}_{i\in\omega}$ is a \emph{backward
irreducible (or irreducible) sequence w.r.t.\ }$
\mathscr{C}
$ if $\{\mathfrak{F}_{i}\}_{i\in\omega}$ is backward irreducible (or
irreducible) and $\mathfrak{F}_{i}\in
\mathscr{C}
$ for each $i\in\omega$. A modal logic $\mathbf{L}$ is \emph{characterized}
by a class $\mathscr{C}$ of frames if $\mathbf{L}=\mathbf{Log}(\mathscr{C})$.

The following theorem provides a sufficient condition of finite axiomatizability
in terms of backward irreducible sequences, and is proved by applying
Theorem \ref{thm: Tarski's criterion of f. ax.}.

\begin{theorem} Let $\mathbf{L}$ be a modal logic, and let $\mathscr{C}$
be a class of frames for $\mathbf{L}$ such that each extension of
$\mathbf{L}$ is characterized by a subclass of $\mathscr{C}$. Then
all extensions of $\mathbf{L}$ are finitely axiomatizable over $\mathbf{L}$
if there is no backward irreducible sequence w.r.t.\ $\mathscr{C}$.\label{thm: all exts of L are fa, iff}
\end{theorem}

\begin{proof} Suppose that $\mathbf{L}^{\prime}$ extends $\mathbf{L}$
but is not finitely axiomatizable over $\mathbf{L}$. By Theorem \ref{thm: Tarski's criterion of f. ax.},
there is an infinite ascending $\subset$-chain $\mathbf{L}_{0}\subset\mathbf{L}_{1}\subset\cdots$
of extensions of $\mathbf{L}$, and then for each $i\in\omega$, there
is a $\phi_{i}\in\mathbf{L}_{i+1}-\mathbf{L}_{i}$, and hence by hypothesis,
$\mathfrak{F}_{i}\nvDash\phi_{i}$ for a member $\mathfrak{F}_{i}$
of $
\mathscr{C}
$ such that $\mathfrak{F}_{i}\vDash\mathbf{L}_{i}$, which implies
that for each $i,j\in\omega$ with $i<j$, $\mathfrak{F}_{i}\nvDash\phi_{i}$
and $\mathfrak{F}_{j}\nvDash\phi_{i}$. Therefore, $\{\mathfrak{F}_{i}\}_{i\in\omega}$
is a backward irreducible sequence w.r.t.\ $\mathscr{C}$.\end{proof}

From now on, whenever we speak of an backward irreducible (or irreducible)
sequence \emph{of such and such frames} (\emph{for} $\mathbf{L}$),
we mean an backward irreducible (or irreducible) sequence w.r.t.\ the
class of such and such frames (for $\mathbf{L}$). The following
corollary is often applied in studies of finite axiomatizability of
modal logics whose extensions have the f.m.p. (see, e.g., \cite{Fine-71-logic-containing-S4-3},
\cite{Nagle-k5-1} and \cite{xu-ext-K4.3})

\begin{corollary} Let $\mathbf{L}$ be any modal logic whose extensions
all have the f.m.p. Then all extensions of $\mathbf{L}$ are finitely
axiomatizable over $\mathbf{L}$ if there is no backward irreducible
sequence of finite rooted frames for $\mathbf{L}$.\label{thm: ext of L f.a. iff no funny seq 4 L}
\end{corollary}

\begin{proof} Let $
\mathscr{C}
$ be the class of finite rooted frames for $\mathbf{L}$. It follows
from hypothesis that each extension of $\mathbf{L}$ is characterized
by a subclass of $
\mathscr{C}
$. Hence the conclusion follows from Theorem \ref{thm: all exts of L are fa, iff}.
\end{proof}

In the following, we prove the converse of Corollary \ref{thm: ext of L f.a. iff no funny seq 4 L},
and combined it with Corollary \ref{thm: ext of L f.a. iff no funny seq 4 L}
to get our final criterion of finite axiomatizability in terms of
(backward) irreducible sequences. Let $\mathfrak{F}=\left\langle W,R\right\rangle $
be a finite rooted transitive frame, where $W=\{w_{0},\ldots,w_{n}\}$
with $w_{0}$ to be a root of $\mathfrak{F}$, and $w_{0},\ldots,w_{n}$
to be all distinct. We call $\left\langle w_{0},\ldots,w_{n}\right\rangle $
an \emph{ordering of points in} $\mathfrak{F}$. Let $p_{0},\ldots,p_{n}$
be distinct propositional letters, and let us call a conjunction of
the following formulas a \emph{frame formula for} $\mathfrak{F}$
\emph{w.r.t.}\ $\left\langle w_{0},\ldots,w_{n}\right\rangle $:
\begin{itemize}
\item $p_{0}$,
\item $\Box(p_{0}\vee\cdots\vee p_{n})$,
\item $
{\textstyle \bigwedge}
\{(p_{i}\rightarrow\lnot p_{j})\wedge\Box(p_{i}\rightarrow\lnot p_{j}):i,j\leqslant n$ and $i\neq j\}$,
\item $
{\textstyle \bigwedge}
\{(p_{i}\rightarrow\Diamond p_{j})\wedge\Box(p_{i}\rightarrow\Diamond p_{j}):i,j\leqslant n$ and $Rw_{i}w_{j}\}$,
\item $
{\textstyle \bigwedge}
\{(p_{i}\rightarrow\lnot\Diamond p_{j})\wedge\Box(p_{i}\rightarrow\lnot\Diamond p_{j}):i,j\leqslant n$ and not $Rw_{i}w_{j}\}$.
\end{itemize}
\noindent A \emph{frame formula}\footnote{A frame formula for $\mathfrak{F}$ is also known as a \emph{Jankov-Fine
formula} for $\mathfrak{F}$ (see \cite{Blackburn-de-Rijke-ml-book}).
The term ``frame formula”\ goes back to \cite{Fine-74-ascending-chain-S4-logics}.}\emph{ for} $\mathfrak{F}$ is a frame formula for $\mathfrak{F}$
w.r.t.\ an ordering $\left\langle u_{0},\ldots,u_{n}\right\rangle $
of points in $\mathfrak{F}$, where $u_{0}$ is a root. 

\begin{lemma} \label{Claim-JF-Sat_at_root}Let $\mathfrak{F}=\left\langle W,R\right\rangle $
be a finite rooted transitive frame, for which $\phi$ is a frame
formula w.r.t.\ an ordering $\left\langle w_{0},\ldots,w_{n}\right\rangle $
of points in $\mathfrak{F}$.\ Then $\phi$ is  satisfiable in $\mathfrak{F}$
at its root $w_{0}$. \end{lemma}

\begin{proof} Let $\mathfrak{M}=\left\langle \mathfrak{F},V\right\rangle $
where $V(p_{i})=\{w_{i}\}$ for each $i\leqslant n$. It is routine
to check that $\mathfrak{M},w_{0}\vDash\phi$. \end{proof}

The following is Lemma 3.20 from \cite{Blackburn-de-Rijke-ml-book},
and the proof is left to the reader.

\begin{lemma} \label{lem-JF-Property}Let $\mathfrak{F}$ be a finite
rooted transitive frame, for which $\phi$ is a frame formula, and
let $\mathfrak{G}=\left\langle U,S\right\rangle $ be any transitive
frame with $u\in U$. Then $\phi$ is satisfiable in $\mathfrak{G}$
at $u$ iff $\mathfrak{G}|_{u}$ is reducible to $\mathfrak{F}$.\end{lemma}

\begin{proposition} Let $\{\mathfrak{F}_{i}\}_{i\in\omega}$ be an
irreducible sequence of finite rooted transitive frames. Then there
is a continuum of extensions of $\mathbf{L}=\mathbf{Log}(\{\mathfrak{F}_{i}\}_{i\in\omega})$.\label{lem: (Fn) is an odd seq}
\end{proposition}

\begin{proof} For each $i\in\omega$, let $\phi_{i}$ be a frame
formula for $\mathfrak{F}_{i}$; and for each $I\subseteq\omega$,
let $\mathbf{L}_{I}=\mathbf{Log}(\{\mathfrak{F}_{i}\}_{i\in I})$.
Consider any $I,J\subseteq\omega$ such that there is an $i\in I-J$.
For each $k\in J$, because $i\neq k$, $\phi_{i}$ is by hypothesis
and Lemma \ref{lem-JF-Property} not satisfiable in $\mathfrak{F}_{k}$,
and hence $\mathfrak{F}_{k}\vDash\lnot\phi_{i}$. It then follows
that $\lnot\phi_{i}\in\mathbf{L}_{J}$. By Lemma \ref{Claim-JF-Sat_at_root},
$\mathfrak{F}_{i}\nvDash\lnot\phi_{i}$, and then $\lnot\phi_{i}\notin\mathbf{L}_{I}$,
and hence $\mathbf{L}_{I}\neq\mathbf{L}_{J}$. A similar argument
shows that $\mathbf{L}_{I}\neq\mathbf{L}_{J}$ if there is a $j\in J-I$.
Hence $\mathbf{L}_{I}\neq\mathbf{L}_{J}$ for all $I,J\subseteq\omega$
such that $I\neq J$. It then follows that there is a continuum of
extensions of $\mathbf{L}$. \end{proof}

For each frame $\mathfrak{F}=\left\langle W,R\right\rangle $, we
use $\left\Vert \mathfrak{F}\right\Vert $ for $\left\vert W\right\vert $.
The following is easily verifiable:

\begin{fact} \label{fact:infinite-iso-subsequence}Let $\{\mathfrak{F}_{i}\}_{i\in\omega}$
be an infinite sequence of frames such that for an $m\geqslant1$,
$\left\Vert \mathfrak{F}_{i}\right\Vert \leqslant m$ for all $i\in\omega$.
Then there is an infinite $I\subseteq\omega$ such that all frames
in $\{\mathfrak{F}_{i}\}_{i\in I}$ are isomorphic. \end{fact}

\begin{proposition} Each infinite backward irreducible sequence of
finite  frames has an infinite irreducible subsequence.\label{thm: odd seq =00003D> inter-irreducible subseq}
\end{proposition}

\begin{proof} Let $\{\mathfrak{F}_{i}\}_{i\in\omega}$ be a backward
irreducible sequence of finite  frames. By Fact \ref{fact:infinite-iso-subsequence},
there is no $m\in\omega$ such that $\left\Vert \mathfrak{F}_{i}\right\Vert \leqslant m$
for all $i\in\omega$. Then there is an infinite $I\subseteq\omega$
such that for all $i,j\in I$ with $i<j$, $\left\Vert \mathfrak{F}_{i}\right\Vert <\left\Vert \mathfrak{F}_{j}\right\Vert $,
and hence no point-generated subframe of $\mathfrak{F}_{i}$ is reducible
to $\mathfrak{F}_{j}$. It then follows that $\{\mathfrak{F}_{i}\}_{i\in I}$
is irreducible. \end{proof}

\begin{theorem} Let $\mathbf{L}$ be a transitive logic whose extensions
all have the f.m.p. Then the following are equivalent:\footnote{\label{fn:stronger fa iff}Since a continuum of extensions of $\mathbf{L}$
can be constructed from an infinite irreducible sequence of finite
rooted frames for $\mathbf{L}$, we also have the following equivalences:
there is a continuum of non-finitely-axiomatizble extensions of $\mathbf{L}$
iff there is an infinite backward irreducible sequence of finite rooted
frames for $\mathbf{L}$ iff there is an infinite irreducible sequence
of finite rooted frames for $\mathbf{L}$. }\label{coro: trans L fa iff}
\begin{enumerate}
\item all extensions of $\mathbf{L}$ are finitely axiomatizable over $\mathbf{L}$;
\item there is no infinite backward irreducible sequence of finite rooted
frames for $\mathbf{L}$;
\item there is no infinite irreducible sequence of finite rooted frames
for $\mathbf{L}$. 
\end{enumerate}
\end{theorem}

\begin{proof} By definition of irreducible sequences and Proposition
\ref{thm: odd seq =00003D> inter-irreducible subseq}, (ii) is equivalent
to (iii). According to Corollary \ref{thm: ext of L f.a. iff no funny seq 4 L}
and Proposition \ref{lem: (Fn) is an odd seq}, we have that (ii)
implies (i) and (i) implies (iii), and hence (i) is equivalent to
(ii).\end{proof}

\section{Transitive Logics of Finite Depth and of Finite Weak Width\label{sec:Transitive-Logics-of FD and FSW}}

In this section, we present weak width formulas $\mathsf{Wid}_{n}^{+}$
($n\geqslant1$), discuss their frame conditions, and then show that
there are non-finitely-axiomatizble extensions of $\mathbf{K4B}_{n}\oplus\mathsf{Wid}_{k}^{+}$
whenever $n\geqslant3$ and $k\geqslant2$. 

For each $n\geqslant1$, let $\mathsf{Wid}_{n}^{+}$ be the following
formula:
\[
\mathsf{Wid}_{n}^{+}=q\wedge\Diamond(\Box\lnot q\wedge({
{\textstyle \bigwedge\nolimits _{i\leqslant n}}
}\Diamond p_{i}))\rightarrow{
{\textstyle \bigvee\nolimits _{0\leqslant i\neq j\leqslant n}}
}\Diamond(p_{i}\wedge(p_{j}\vee\Diamond p_{j}))\text{.}
\]
 A transitive logic is \emph{of weak width} $n$ ($n\geqslant1$) if
it contains $\mathsf{Wid}_{n}^{+}$ but not $\mathsf{Wid}_{n-1}^{+}$,
and is \emph{of finite weak width} if it contains $\mathsf{Wid}_{k}^{+}$
for a $k\geqslant1$.

\begin{proposition}\label{prop: weak width frame condition}Let $\mathfrak{F}=\left\langle W,R\right\rangle $
be a transitive frame, and let $w\in W$ and $n\geqslant1$. Then
$\mathfrak{F},w\vDash\mathsf{Wid}_{n}^{+}$ iff for each $u$ with
$\vec{R}wu$, $\mathfrak{F}|_{u}$ is of width at most $n$. \end{proposition}

\begin{proof} Suppose that $\mathfrak{M},w\nvDash\mathsf{Wid}_{n}^{+}$
for a model $\mathfrak{M}$ on $\mathfrak{F}$. Because $\mathfrak{M},w\vDash q\wedge\Diamond(\Box\lnot q\wedge({\bigwedge\nolimits _{i\leqslant n}}\Diamond p_{i}))$,
there is a $u\in\left.w\right\uparrow $ such that $\mathfrak{M},u\vDash\Box\lnot q\wedge({\bigwedge\nolimits _{i\leqslant n}}\Diamond p_{i})$,
and then $\vec{R}wu$, and for each $i\leqslant n$, $\mathfrak{M},v_{i}\vDash p_{i}$
for a $v_{i}\in\left.u\right\uparrow $. Consider any $i,j\leqslant n$
such that $i\neq j$. Because $\mathfrak{M},w\nvDash{\bigvee\nolimits _{0\leqslant i\neq j\leqslant n}}\Diamond(p_{i}\wedge(p_{j}\vee\Diamond p_{j}))$,
and because $Rwv_{i}$ by the transitivity of $R$, it then follows
that $\mathfrak{M},v_{i}\vDash p_{i}$ and $\mathfrak{M},v_{i}\nvDash p_{j}\vee\Diamond p_{j}$,
and $\mathfrak{M},v_{j}\vDash p_{j}$ and $\mathfrak{M},v_{j}\nvDash p_{i}\vee\Diamond p_{i}$,
and then neither $v_{i}=v_{j}$ nor $Rv_{i}v_{j}$ nor $Rv_{j}v_{i}$.
Hence $\{v_{0},\ldots,v_{n}\}$ is an antichain, and then $\mathfrak{F}|_{u}$
is of width greater than $n$ because $\{v_{0},\ldots,v_{n}\}\subseteq\left.u\right\uparrow $.

Suppose that there is a $u\in\mathbf{c}_{(w)}{\uparrow}^{-}$ such
that $\mathfrak{F}|_{u}$ is of width greater than $n$. Then there
is an antichain $\{v_{0},\ldots,v_{n}\}\subseteq\left.u\right\uparrow $.
Let $\mathfrak{M}=\left\langle \mathfrak{F},V\right\rangle $ where
$V(q)=\{w\}$, and $V(p_{i})=v_{i}$ for each $i\leqslant n$. Since
$\vec{R}wu$ and $\{v_{0},\ldots,v_{n}\}\subseteq\left.u\right\uparrow $,
it is easy to see that $\mathfrak{M},u\vDash\Box\lnot q\wedge({\bigwedge\nolimits _{i\leqslant n}}\Diamond p_{i})$,
and then $\mathfrak{M},w\vDash q\wedge\Diamond(\Box\lnot q\wedge({\bigwedge\nolimits _{i\leqslant n}}\Diamond p_{i}))$.
For each $v\in\left.w\right\uparrow $ and each $i\leqslant n$, if
$\mathfrak{M},v\vDash p_{i}$, we know by definition of $V$ that
$v=v_{i}$ and $\mathfrak{M},v\nvDash p_{j}\vee\Diamond p_{j}$ for
each $j\leqslant n$ with $j\neq i$. Hence for each $v\in\left.w\right\uparrow $,
$\mathfrak{M},v\nvDash{\bigvee\nolimits _{0\leqslant i\neq j\leqslant n}(}p_{i}\wedge(p_{j}\vee\Diamond p_{j}))$,
from which it follows that $\mathfrak{M},w\nvDash{\bigvee\nolimits _{0\leqslant i\neq j\leqslant n}\Diamond}(p_{i}\wedge(p_{j}\vee\Diamond p_{j}))$,
and hence $\mathfrak{M},w\nvDash\mathsf{Wid}_{n}^{+}$. \end{proof}

In what follows, we show that there are non-finitely-axiomatizble
extensions of $\mathbf{K4B}_{n}\oplus\mathsf{Wid}_{k}^{+}$ whenever
$n\geqslant3$ and $k\geqslant2$, by way of constructing irreducible
sequences of finite rooted transitive frames of rank $3$. 

We now construct irreducible sequences of finite rooted transitive
frames of rank $3$, in each of which all points of rank $2$ have
exactly two proper successors. For each $n\in\omega$, let $B_{n}=\{X\in
\mathscr{P}
(C_{n}):\left\vert X\right\vert =2\}$, where $C_{n}=\{k:k\leqslant n+1\}$, and let $\mathfrak{H}_{n}=\left\langle W_{n},E_{n}\right\rangle $,
where
\begin{align*}
 & W_{n}=\{a\}\cup B_{n}\cup C_{n}\text{,}\\
 & E_{n}=\{\left\langle a,u\right\rangle :u\in B_{n}\cup C_{n}\}\cup\{\left\langle b,c\right\rangle \in B_{n}\times C_{n}:c\in b\}\text{.}
\end{align*}
It is easy to see that for each $n\in\omega$ and in each of $\mathfrak{H}_{n}$,
$a$ is of rank $3$, and members of $B_{n}$ are of rank $2$ while
those of $C_{n}$ are of rank $1$. Note that for each $n\in\omega$,
$\mathfrak{H}_{n}$ is a finite strict partial order. Since all points
of rank 2 in these frames have exactly two proper successors, the
following Fact holds:

\begin{fact} For each $n\geqslant2$, $\mathsf{Wid}_{n}^{+}$ is
valid in all members of $\{\mathfrak{H}_{n}\}_{n\in\omega}$.\label{coro: SWn valid all Fi bt Gi and G'i}
\end{fact}

In our proof of Lemma \ref{lem: bad sequence 2}, we make use of the
following simple fact about reduction:

\begin{fact} \label{fact:reduction rank > rank}Let $f$ be a reduction
of $\mathfrak{F}$ to $\mathfrak{G}$, where both $\mathfrak{F}$
and $\mathfrak{G}$ are transitive, and let $w$ be a point in $\mathfrak{F}$.
Then the following hold:
\begin{enumerate}
\item $w$ is a dead-end in $\mathfrak{F}$ iff $f(w)$ is a dead-end in
$\mathfrak{G}$;\label{fact:reduction rank > rank, 1}
\item for each $n\geqslant1$, if $f(w)$ is of rank $n$ in $\mathfrak{G}$,
then $w$ is of rank at least $n$ in $\mathfrak{F}$.\label{fact:reduction rank > rank, 2} 
\end{enumerate}
\end{fact}

\begin{lemma} $\{\mathfrak{H}_{n}\}_{n\in\omega}$ is irreducible.\label{lem: bad sequence 2}
\end{lemma}

\begin{proof} Let $k,n\in\omega$ with $k<n$. We only show that
$\mathfrak{\mathfrak{H}}_{n}$ is not reducible to $\mathfrak{\mathfrak{H}}_{k}$,
the other direction is trivial because $\left\vert W_{k}\right\vert <\left\vert W_{n}\right\vert $.
Let us use $R$ for $E_{n}$ and $S$ for $E_{k}$. By definition,
$\left.b\right\uparrow _{E_{n}}=b$ for each $b\in B_{n}$, and hence
by hypothesis,
\begin{equation}
\left.b\right\uparrow _{R}=b\text{ for each }b\in B_{n}\text{.}\label{Display:O2}
\end{equation}
Suppose for reductio that $f$ reduces $\mathfrak{\mathfrak{H}}_{n}$
to $\mathfrak{\mathfrak{H}}_{k}$. It follows from Fact \ref{fact:reduction rank > rank}
that $f(a)=a$, $f[B_{n}]=B_{k}$ and $f[C_{n}]=C_{k}$. Since $k<n$,
$C_{k}\subset C_{n}$, and then there are distinct $c,c^{\prime}\in C_{n}$
such that $f(c)=f(c^{\prime})$. Let $b=\{c,c^{\prime}\}\in B_{n}$.
Then $f(b)=\{v,v^{\prime}\}\in B_{k}$ for some distinct $v,v^{\prime}\in C_{k}$.
By definition,
\begin{equation}
Sf(b)v\text{, }Sf(b)v^{\prime}\text{ and }f(b)\neq v,v^{\prime}\text{.}\label{display:03}
\end{equation}
Since $f(c)=f(c^{\prime})$, either $v\neq f(c),f(c^{\prime})$ or
$v^{\prime}\neq f(c),f(c^{\prime})$.\ If $v\neq f(c),f(c^{\prime})$,
then by (\ref{Display:O2}) and (\ref{display:03}), $Sf(b)v$ but
$f(u)\neq v$ for each $u\in\left.b\right\uparrow _{R}=\{c,c^{\prime}\}$,
contrary to the supposition that $f$ reduces $\mathfrak{\mathfrak{H}}_{n}$
to $\mathfrak{\mathfrak{H}}_{k}$. By the same token, if $v^{\prime}\neq f(c),f(c^{\prime})$,
then $Sf(b)v^{\prime}$ but $f(u)\neq v^{\prime}$ for each $u\in\left.b\right\uparrow _{R}$,
contrary to the supposition again. \end{proof}

\begin{theorem} Let $n\geqslant3$ and $k\geqslant2$. There are
non-finitely-axiomatizble extensions of $\mathbf{K4B}_{n}\oplus\mathsf{Wid}_{k}^{+}$.\footnote{According to footnote \ref{fn:stronger fa iff}, we can actually show
that there is a continuum of extensions of $\mathbf{K4B}_{n}\oplus\mathsf{Wid}_{k}^{+}$
whenever $n\geqslant3$ and $k\geqslant2$.} \end{theorem}

\begin{proof} By Proposition \ref{prop: frame cond 4 B_k} and Fact
\ref{coro: SWn valid all Fi bt Gi and G'i}, we have that for each
$n\in\omega$, $\mathfrak{H}_{n}$ is a frame for $\mathbf{K4B}_{n}\oplus\mathsf{Wid}_{k}^{+}$.
It then follows from Lemma \ref{lem: bad sequence 2} and Theorem
\ref{coro: trans L fa iff}, there are non-finitely-axiomatizble extensions
of $\mathbf{K4B}_{n}\oplus\mathsf{Wid}_{k}^{+}$. \end{proof}

\section{Finite Axiomatizability of Transitive Logics of Finite Depth and
of Weak Width $1$\label{sect: FA of lgcs of finite suc-eq-width}}

Consider the following formulas, where $n\geqslant1$:
\[
\mathsf{Wid}_{n}^{\bullet}={
{\textstyle \bigwedge\nolimits _{i\leqslant n}}
}\Diamond(p_{i}\wedge\Box\neg p_{i})\rightarrow{
{\textstyle \bigvee\nolimits _{0\leqslant i\neq j\leqslant n}}
}\Diamond(p_{i}\wedge(p_{j}\vee\Diamond p_{j}))\text{.}
\]
In this section, we discuss the frame conditions for $\mathsf{Wid}_{n}^{\bullet}$
with $n\geqslant1$, provide a study of well-quasi-orders on trees,
and then prove the finite axiomatizability of each transitive logic
of finite depth and of weak width $1$ that contains $\mathsf{Wid}_{n}^{\bullet}$
for an $n\geqslant1$.

\subsection{Transitive Frames for $\mathsf{Wid}_{n}^{\bullet}$ \label{sect: frames of finite suc-eq-width}}

Let $\mathfrak{F}=\left\langle W,R\right\rangle $ be any frame, and
let $A$ be an antichain in $\mathfrak{F}$. We say $A$ is \emph{irreflexive}
if for all $w\in A$, $Rww$ fails. 

\begin{proposition} Let $\mathfrak{F}=\left\langle W,R\right\rangle $
be any transitive frame, and let $w\in W$ and $n\geqslant1$. Then
$\mathfrak{F},w\vDash\mathsf{Wid}_{n}^{\bullet}$ iff $\left\vert A\right\vert \leqslant n$
for each irreflexive antichain $A$ in $\mathfrak{F}|_{w}$.\label{prop: frame cond 4 Wid*n}
\end{proposition}

\begin{proof} Suppose that $\mathfrak{M},w\nvDash\mathsf{Wid}_{n}^{\bullet}$
for a model $\mathfrak{M}$ on $\mathfrak{F}$. Because $\mathfrak{M},w\vDash{
{\textstyle \bigwedge\nolimits _{i\leqslant n}}
}\Diamond(p_{i}\wedge\Box\neg p_{i})$, we have that for each $i\leqslant n$, $\mathfrak{M},u_{i}\vDash p_{i}$
for an irreflexive point $u_{i}\in\left.w\right\uparrow $. Consider
any $i,j\leqslant n$ such that $i\neq j$. Because $\mathfrak{M},w\nvDash{\bigvee\nolimits _{0\leqslant i\neq j\leqslant n}}\Diamond(p_{i}\wedge(p_{j}\vee\Diamond p_{j}))$,
it then follows from $Rwu_{i}$ and $\mathfrak{M},u_{i}\vDash p_{i}$
that $\mathfrak{M},u_{i}\nvDash p_{j}\vee\Diamond p_{j}$; it further
follows from $Rwu_{j}$ and $\mathfrak{M},u_{j}\vDash p_{j}$ that
$\mathfrak{M},u_{j}\vDash p_{j}$ and $\mathfrak{M},u_{j}\nvDash p_{i}\vee\Diamond p_{i}$.
So we have that neither $u_{i}=u_{j}$ nor $Ru_{i}u_{j}$ nor $Ru_{j}u_{i}$.
Hence $\{u_{0},\ldots,u_{n}\}$ is an irreflexive antichain in $\mathfrak{F}|_{w}$
whose cardinality is greater than $n$.

Suppose that there is an irreflexive antichain $\{u_{0},\ldots,u_{n}\}$
in $\mathfrak{F}|_{w}$. Let $\mathfrak{M}=\left\langle \mathfrak{F},V\right\rangle $
where $V(p_{i})=u_{i}$ for each $i\leqslant n$. It is easy to see
that $\mathfrak{M},u_{i}\vDash p_{i}\wedge\Box\neg p_{i}$ for each
$i\leqslant n$, and hence $\mathfrak{M},w\vDash{
{\textstyle \bigwedge\nolimits _{i\leqslant n}}
}\Diamond(p_{i}\wedge\Box\neg p_{i})$. For each $v\in\left.w\right\uparrow $ and each $i\leqslant n$,
if $\mathfrak{M},v\vDash p_{i}$, we know by definition of $V$ that
$v=u_{i}$ and $\mathfrak{M},v\nvDash p_{j}\vee\Diamond p_{j}$ for
each $j\leqslant n$ with $j\neq i$. Hence for each $v\in\left.w\right\uparrow $,
$\mathfrak{M},v\nvDash{\bigvee\nolimits _{0\leqslant i\neq j\leqslant n}(}p_{i}\wedge(p_{j}\vee\Diamond p_{j}))$,
from which it follows that $\mathfrak{M},w\nvDash{\bigvee\nolimits _{0\leqslant i\neq j\leqslant n}\Diamond}(p_{i}\wedge(p_{j}\vee\Diamond p_{j}))$,
and hence $\mathfrak{M},w\nvDash\mathsf{Wid}_{n}^{\bullet}$. \end{proof}

The following proposition is a direct consequence of Proposition \ref{prop: frame cond 4 Wid*n}.

\begin{proposition} For each rooted transitive frame $\mathfrak{F}$
and each $n\geqslant1$, $\mathfrak{F}\vDash\mathsf{Wid}_{n}^{\bullet}$
iff $\left\vert A\right\vert \leqslant n$ for each irreflexive antichain
$A$ in $\mathfrak{F}$.\label{prop: frame cond 4 Wid_n-1} \end{proposition}

\subsection{Well-quasi-orders}

Let $A$ be any set. A binary relation $R$ on $A$ is a \emph{quasi-order
}iff it is reflexive and transitive. Let $\preceq$ be a quasi-order
on $A$. We say $\preceq$ is a \emph{well-quasi-order} (in short:
\emph{wqo}) iff every infinite sequence $(a_{k})_{k\in\omega}$ of
elements of $A$ contains an infinite subsequence $(a_{k})_{k\in I\subseteq\omega}$
of it such that $a_{i}\preceq a_{j}$ for all $i,j\in I$ with $i<j$.\footnote{Another well-known definition of \emph{well-quasi-order} is as follows:
$\preceq$ is a \emph{well-quasi-order} iff every infinite sequence
$(a_{k})_{k\in\omega}$ of elements of $A$ contains two element $a_{i},a_{j}$
such that $a_{i}\preceq a_{j}$ with $i<j$. These two definitions
are equivalent, and a proof of their equivalence can be found in Lemma
2.5 in \cite{Gallier1991What}.} Note that any quasi-order on $A$ is \emph{wqo} if $A$ is finite,
and that $\preceq$ is a \emph{wqo} on any $A'\subseteq A$ if $\preceq$
is a \emph{wqo} on $A$. Let $\leqslant$ be the usual less-than-order
on $\omega$. We fix a new order $\preccurlyeq$ on $\omega$ as follows:
$m\preccurlyeq n$ iff either $m=n=0$ or $0<m\leqslant n$. 

\begin{fact} Both $\leqslant$ and $\preccurlyeq$ are \emph{wqo}
on $\omega$. \label{fact: infinite seq =00003D> inf sub seq, black chain-1}
\end{fact}

The following lemma is from \cite{Nash1987On}, and the reader can
also refer to Lemma 2.6 in \cite{Gallier1991What}.

\begin{lemma}\label{lem:card product wqo}Let $\preceq_{1}$ and
$\preceq_{2}$ be wqo on set $A_{1}$ and $A_{2}$ respectively, and
let $\preceq$ be the order on $A_{1}\times A_{2}$ defined as follows:
$\left\langle a_{1},a_{2}\right\rangle \preceq\left\langle a'_{1},a'_{2}\right\rangle \text{ iff }a_{1}\preceq_{1}a'_{1}\text{ and }a_{2}\preceq_{2}a'_{2}.$
Then $\preceq$ is a wqo on $A_{1}\times A_{2}$.

\end{lemma}

Let $A$ be any set. We use $A^{*}$ for set of all finite sequences
(or strings) over $A$, use $\ell(s)$ for the length\emph{ }of the
sequence $s$, and for each $i\leqslant k$, we will use $\#_{i}(s)$
for the $i$-th member of $s$, starting from 0. For each $n\geqslant0$,
we fix $\mathsf{Seq}_{\leqslant n}(A)=\{s\in A^{*}:\ell(s)\leqslant n\}$.
Let $\preceq$ be a quasi-order on $A$. We define the orders $\trianglelefteq$
and $\ll$ on $A^{*}$ as follows: 
\begin{itemize}
\item for all $s,t\in A^{*}$, $s\trianglelefteq t$ iff $\ell(s)=\ell(t)$,
and for each $i<\ell(s)$, $\#_{i}(s)\preceq\#_{i}(t)$.\label{def: black triangle rel-1} 
\item for all $s,t\in A^{*}$ where $s=(a_{i})_{i<k}$ and $t=(b_{i})_{i<n}$,
$t\ll s$ iff either $n=k=0$, or $n\geqslant k>0$ and $a_{k}\preceq b_{n}$
and $s\trianglelefteq t^{\prime}$ for a subsequence $t^{\prime}$
of $t$.
\end{itemize}
It is easy to see that both $\trianglelefteq$ and $\mathbf{\ll}$
are quasi-orders on $A^{*}$. Furthermore, Lemma \ref{lem:card product wqo}
can be applied to show the following Lemma by a trivial induction.

\begin{lemma} If $\preceq$ is a wqo on $A$, then \textbf{$\trianglelefteq$}
is a wqo on $\mathsf{Seq}_{\leqslant n}(A)$ for all $n\geqslant0$.\label{prop: infinite seq =00003D> inf sub seq, black chain-2}
\end{lemma}

The following theorem is a slightly stronger formulation of Theorem
3.2 in \cite{Gallier1991What}, however the same proof can be applied
here. A restricted version of the theorem, where $A$ is the set of
natural number, is proved in \cite{Fine-71-logic-containing-S4-3}
along the same line as \cite{Gallier1991What}. 

\begin{theorem} If $\preceq$ is a wqo on $A$, then \textbf{$\ll$}
is a wqo on $A^{*}$.\label{lem: inf seq has 2-chain as sub seq-1}
\end{theorem}

A \emph{tree} is a pair $\left\langle T,\leq\right\rangle $, in which
$T$ is a nonempty set and $\leq$ is a partial ordering on $T$ satisfying
\emph{downward connectedness} ($\forall m\forall m^{\prime}\exists w(w\leq m\wedge w\leq m^{\prime})$)
and \emph{no downward branching} ($\forall m\forall w\forall w^{\prime}(w\leq m\wedge w^{\prime}\leq m\rightarrow w\leq w^{\prime}\vee w^{\prime}\leq w)$).
$w<u$ is introduced as $w\leq u\wedge w\neq u$. Let $\mathfrak{T}=\left\langle T,\leq\right\rangle $
be any tree. Note that the set $anc_{\mathfrak{T}}(w)=\{u\in T:u\leq w\}$
is a chain under $\leq$, and a finite tree always has a unique root.
We use $dom(\mathfrak{T})$ for the domain of $\mathfrak{T}$, and
use $root(\mathfrak{T})$ for the root of $\mathfrak{T}$ when it
exists. For any $w\in T$, \emph{the level of} $w$ in $\mathfrak{T}$
is $lev_{\mathfrak{T}}(w)=\left|anc_{\mathfrak{T}}(w)\right|$, the
\emph{set of immediate successors} of $w$ is $suc_{\mathfrak{T}}(w)=\{u\in T:w<u\wedge\neg\exists v(w<v<u)\}$,
and the \emph{height} of $\mathbf{\mathfrak{T}}$ is $heit(\mathfrak{T})=max\{lev_{\mathbf{\mathfrak{T}}}(w):w\in T\}$.
Given a set $\Sigma$ of labels, a \emph{$\Sigma$-tree }is a pair
$\left\langle \mathfrak{T},\tau\right\rangle $, where $\mathfrak{T}$
is a tree and $\tau$ is a labeling function on $\mathfrak{T}$ from
$dom(\mathfrak{T})$ to $\Sigma$. Let $\mathfrak{t}=\left\langle \mathfrak{T},\tau\right\rangle $\emph{
}be any\emph{ $\Sigma$}-tree\emph{ }where $\mathfrak{T}=\left\langle T,\leq\right\rangle $\emph{.
A $\Sigma$}-tree\emph{ $\mathfrak{t}$} is \emph{finite }if its underlying
tree $\mathfrak{T}$ is finite, and the \emph{height (domain, root,
etc.) }can be level up to \emph{$\Sigma$}-trees from their underlying
trees naturally. For each $\Delta\subseteq\Sigma$, $dom(\mathfrak{t})^{\Delta}=\{w\in dom(\mathfrak{t}):\tau(w)\in\Delta\}$,
and we use $dom(\mathfrak{t})^{l}$ for $dom(\mathfrak{t})^{\{l\}}$. 

In the following, we consider only finite \emph{$\omega$}-trees,
and use $\mathbf{T^{\omega}}$ for the set of all finite \emph{$\omega$}-trees.
For each $m,n\geqslant1$, we fix 
\begin{align*}
 & \mathbf{T}_{=m,<n}^{\omega}=\{\mathfrak{t}\in\mathbf{T}^{\omega}:heit(\mathfrak{t})=m\wedge\left|dom(\mathfrak{t})^{0}\right|<n\},\\
 & \mathbf{T}_{\leqslant m,<n}^{\omega}=\bigcup_{1\leqslant i\leqslant n}\mathbf{T}_{=m,<n}^{\omega}.
\end{align*}
Note that $\mathbf{T}_{=1,<n}^{\omega}=\mathbf{T}_{\leqslant1,<n}^{\omega}$
and all $\omega$-trees in them have only one node, i.e. the root.
It is convenient for our discussion to represent a \emph{$\Sigma$}-tree
$\mathfrak{t}=\left\langle \mathfrak{T},\tau\right\rangle $\emph{
}as the following triple: 
\begin{equation}
\mathfrak{t}=\left\langle \left(root(\mathfrak{t}),\tau(root(\mathfrak{t}))\right),\left(\mathfrak{t}_{1},\ldots,\mathfrak{t}_{m}\right),\left(\mathfrak{t}_{m+1},\ldots,\mathfrak{t}_{m+n}\right)\right\rangle ,\label{eq: tree representation}
\end{equation}
where 
\begin{itemize}
\item $\mathfrak{t}_{1},\ldots,\mathfrak{t}_{m}$ are all subtrees of $\mathfrak{t}$
generated by an element of $\{w\in suc_{\mathfrak{t}}(root(\mathfrak{t})):\tau(w)=0\}$, 
\item $\mathfrak{t}_{m+1},\ldots,\mathfrak{t}_{m+n}$ are all subtrees of
$\mathfrak{t}$ generated by an element of $\{w\in suc_{\mathfrak{t}}(root(\mathfrak{t})):\tau(w)>0\}$,
and 
\item $\tau_{m+n}(root(\mathfrak{t}_{m+n}))=min\{\tau_{i}(root(\mathfrak{t}_{i})):m\leqslant i\leqslant m+n\}$,
in which $\tau_{i}$ is the labeling function in $\mathfrak{t}_{i}$. 
\end{itemize}
We call the triple above a \emph{standard representation triple of}
$\mathfrak{t}$. Note that the last two elements of a standard representation
triple could be the empty sequence, such as when the represented tree
has only one-node. Recall that $m\preccurlyeq n$ iff either $m=n=0$
or $0<m\leqslant n$. We define $\sqsubseteq$ on $\mathbf{T^{\omega}}$
inductively as follows:
\begin{enumerate}
\item for any $\omega$-tree $\mathfrak{t}=\left\langle \left(r,s\right),\left(\right),\left(\right)\right\rangle $
and any $\omega$-tree $\mathfrak{t}'$, $\mathfrak{t}\sqsubseteq\mathfrak{t}'$
iff $\mathfrak{t}'$ is a one-node tree and $s\preccurlyeq\tau'(root(\mathfrak{t}'))$,
where $\tau'$ is the labeling function in $\mathfrak{t}'$;
\item for any $\omega$-tree $\mathfrak{t}=\left\langle \left(r,s\right),\left(\mathfrak{t}_{1},\ldots,\mathfrak{t}_{m}\right),\left(\mathfrak{t}_{m+1},\ldots,\mathfrak{t}_{m+n}\right)\right\rangle $
and any $\omega$-tree $\mathfrak{t}'=\left\langle \left(r',s'\right),\left(\mathfrak{t}'_{1},\ldots,\mathfrak{t}'_{k}\right),\left(\mathfrak{t}'_{k+1},\ldots,\mathfrak{t}'_{k+l}\right)\right\rangle $,
$\mathfrak{t}\sqsubseteq\mathfrak{t}'$ iff $s\preccurlyeq s'$, and 
\begin{enumerate}
\item $m=k$ and for each $1\leqslant i\leqslant m$, $\mathfrak{t}_{i}\sqsubseteq\mathfrak{t}'_{i}$;
\item either $l=n=0$, or $l\geqslant n>0$ and $\mathfrak{t}_{m+n}\sqsubseteq\mathfrak{t}'_{k+l}$
and there are $j_{m+1},\ldots,j_{m+n}$ such that $k+1\leqslant j_{m+1}<\cdots<j_{m+n}\leqslant k+l$,
and $\mathfrak{t}_{h}\sqsubseteq\mathfrak{t}'_{j_{h}}$ for each $h$
with $m+1\leqslant h\leqslant m+n$. 
\end{enumerate}
\end{enumerate}
Note that if we replace $\sqsubseteq$ with $\preceq$ in (a) and
(b), then they become the exactly same as definition of $\trianglelefteq$
and definition of $\mathbf{\ll}$ respectively.

\begin{theorem} For all $m,n\geqslant1$, \textbf{$\sqsubseteq$}
is a wqo on $\mathbf{T}_{\leqslant m,<n}^{\omega}$.\label{thm: wqo on trees}
\end{theorem}

\begin{proof}It suffices to show that for all $m,n\geqslant1$, $\sqsubseteq$
is a wqo on $\mathbf{T}_{=m,<n}^{\omega}$. We prove it by induction
on $m$. The base case ($m=1$) holds because of Fact \ref{fact: infinite seq =00003D> inf sub seq, black chain-1}.
Consider $m=k+1$. Suppose that for all $n\geqslant1$, $\sqsubseteq$
is a wqo on $\mathbf{T}_{=k,<n}^{\omega}$. Let $n\geqslant1$ and
let $(\mathfrak{t}_{i})_{i\in\omega}$ be any infinite sequence of
elements from $\mathbf{T}_{=k+1,<n}^{\omega}$, where $\mathfrak{t}_{i}=\left\langle \left(r_{i},s_{i}\right),\left(\mathfrak{t}_{1}^{i},\ldots,\mathfrak{t}_{m_{i}}^{i}\right),\left(\mathfrak{t}_{m_{i}+1}^{i},\ldots,\mathfrak{t}_{m_{i}+n_{i}}^{i}\right)\right\rangle $
for each $i\in\omega$. We have by Fact \ref{fact: infinite seq =00003D> inf sub seq, black chain-1}
that there is an infinite subsequence $(\mathfrak{t}_{i})_{i\in I_{1}}$
of $(\mathfrak{t}_{i})_{i\in\omega}$ such that $(s_{i})_{i\in I_{1}}$
is an infinite $\preccurlyeq$-chain. Since $\left|dom(\mathfrak{t}_{i})^{0}\right|<n$
for each $i\in\omega$, there is an infinite subsequence $(\mathfrak{t}_{i})_{i\in I_{2}}$
of $(\mathfrak{t}_{i})_{i\in I_{1}}$ such that $m_{i}=m_{j}$ for
all $i,j\in I_{2}$. It then follows from Lemma \ref{fact: infinite seq =00003D> inf sub seq, black chain-1}
and supposition that there is an infinite subsequence $(\mathfrak{t}_{i})_{i\in I_{3}}$
of $(\mathfrak{t}_{i})_{i\in I_{2}}$ such that for each $i<j\in I_{3}$,
$m_{i}=m_{j}$ and for each $1\leqslant h\leqslant m_{i}$, $\mathfrak{t}_{h}^{i}\sqsubseteq\mathfrak{t}_{h}^{j}$.
Apply Theorem \ref{lem: inf seq has 2-chain as sub seq-1} and supposition,
we obtain that there is an infinite subsequence $(\mathfrak{t}_{i})_{i\in I_{4}}$
of $(\mathfrak{t}_{i})_{i\in I_{3}}$ such that for each $i<j\in I_{4}$,
either $n_{i}=n_{j}=0$, or $n_{j}\geqslant n_{i}>0$ and $\mathfrak{t}_{m_{i}+n_{i}}^{i}\sqsubseteq\mathfrak{t}_{m_{j}+n_{j}}^{j}$
and there are $j_{m_{i}+1},\ldots,j_{m_{i}+n_{i}}$ such that $m_{j}+1\leqslant j_{m_{i}+1}<\cdots<j_{m_{i}+n_{i}}\leqslant m_{j}+n_{j}$,
and $\mathfrak{t}_{h}^{i}\sqsubseteq\mathfrak{t}^{j}{}_{j_{h}}$ for
each $h$ with $m_{i}+1\leqslant h\leqslant m_{i}+n_{i}$. By definition
of $\sqsubseteq$, $(\mathfrak{t}_{i})_{i\in I_{4}}$ is an infinite
$\sqsubseteq$-chain, and hence we have that for all $n\geqslant1$,
$\sqsubseteq$ is a wqo on $\mathbf{T}_{=k+1,<n}^{\omega}$. \end{proof}

\subsection{Finite Axiomatizability}

Recall that a transitive logic is of weak width $1$ if it contains
$\mathsf{Wid}_{1}^{+}$. In the subsection, we show the finite axiomatizability
of all transitive logics of finite depth and of finite weak width $1$
that contains $\mathsf{Wid}_{n}^{\bullet}$ for an $n\geqslant1$
(Theorem \ref{thm:K4+Bn+W*k-FA}).

Let $\mathfrak{F}=\left\langle W,R\right\rangle $ be a transitive
frame. The \emph{skeleton of} $\mathfrak{F}$ is $\mathfrak{sk}(\mathfrak{F})=\left\langle \mathfrak{sk}(W),\mathbf{\mathfrak{sk}}(R)\right\rangle $,
where $\mathfrak{sk}(W)$ is the set of clusters in $\mathfrak{F}$,
and for all $\mathbf{c},\mathbf{d}\in\mathfrak{sk}(W)$, $\left\langle \mathbf{c},\mathbf{d}\right\rangle \in\mathbf{\mathfrak{sk}}(R)$
iff $Rwu$ for some $w\in\mathbf{c}$ and $u\in\mathbf{d}$ (in fact,
iff $Rwu$ for all $w\in\mathbf{c}$ and $u\in\mathbf{d}$). For any
binary relation $R$ on a set $W$, we use $R^{*}$ for the reflexive
closure of $R$, i.e., $R\cup\{\left\langle w,w\right\rangle :w\in W\}$,
and use $R^{-1}$ for the inverse of $R$, i.e., $\{\left\langle w,u\right\rangle :\left\langle u,w\right\rangle \in R\}$.
We fix $\mathfrak{sk}(\mathfrak{F})^{*}=\left\langle \mathfrak{sk}(W),\mathbf{\mathfrak{sk}}(R)^{*}\right\rangle $
and $\mathfrak{sk}(\mathfrak{F})^{-1}=\left\langle \mathfrak{sk}(W),(\mathbf{\mathfrak{sk}}(R)^{*})^{-1}\right\rangle $. 

Let $\mathfrak{F}=\left\langle W,R\right\rangle $ be any finite transitive
frame for $\mathsf{Wid}_{1}^{+}$ such that $\mathfrak{sk}(\mathfrak{F})^{-1}$
is a finite tree. The\emph{ representation} \emph{tree} of $\mathfrak{F}$
is the following $\omega$-tree: 
\begin{equation}
\mathfrak{rt}(\mathfrak{F})=\left\langle \mathfrak{sk}(\mathfrak{F})^{-1},\tau\right\rangle ,\label{eq: r.t.}
\end{equation}
where for each $\mathbf{c}\in\mathfrak{sk}(W)$, $\tau(\mathbf{c})=\left\vert \mathbf{c}\right\vert $
if $\mathbf{c}$ is a nondegenerate cluster in $\mathfrak{F}$, otherwise
$\tau(\mathbf{c})=0$. 

\begin{lemma}For any finite transitive frames $\mathfrak{F}$ and
$\mathfrak{G}$ for $\mathsf{Wid}_{1}^{+}$ such that $\mathfrak{sk}(\mathfrak{F})^{-1}$
and $\mathfrak{sk}(\mathfrak{G})^{-1}$ are finite trees, if $\mathfrak{rt}(\mathfrak{F})\sqsubseteq\mathfrak{rt}(\mathfrak{\mathfrak{G}})$,
then $\mathfrak{\mathfrak{G}}$ is reducible to $\mathfrak{F}$.\label{coro: covering function-1}\label{prop:tree-reducible-1}
\end{lemma}

\begin{proof} We prove it by induction on the height of $\mathfrak{rt}(\mathfrak{F})$.
Let $\mathfrak{rt}(\mathfrak{F})=\left\langle \mathfrak{sk}(\mathfrak{F})^{-1},\tau\right\rangle $
and $\mathfrak{rt}(\mathfrak{G})=\left\langle \mathfrak{sk}(\mathfrak{G})^{-1},\sigma\right\rangle $,
and suppose that $\mathfrak{rt}(\mathfrak{F})\sqsubseteq\mathfrak{rt}(\mathfrak{\mathfrak{G}})$.
Consider $heit(\mathfrak{rt}(\mathfrak{F}))=1$. By definition of
$\sqsubseteq$, we have that
\begin{align}
 & heit(\mathfrak{rt}(\mathfrak{G}))=1\text{ and}\label{eq:height =00003D 1}\\
 & \tau(root(\mathfrak{sk}(\mathfrak{F})^{-1}))\preccurlyeq\sigma(root(\mathfrak{sk}(\mathfrak{G})^{-1})).\label{eq: root less-than}
\end{align}
By (\ref{eq:height =00003D 1}), both $\mathfrak{F}$ and $\mathfrak{G}$
are universal frames, i.e., containing only one cluster. Assume that
$\mathbf{c}$ and $\mathbf{d}$ is the unique cluster in $\mathfrak{F}$
and $\mathfrak{G}$, respectively. It follows from (\ref{eq: root less-than})
that $\tau(\mathbf{c})\preccurlyeq\tau(\mathbf{d})$. By definition
of $\preccurlyeq$, either $\tau(\mathbf{c})=\tau(\mathbf{d})=0$
or $0<\tau(\mathbf{c})\leqslant\tau(\mathbf{d})$. If the former holds,
then we have by (\ref{eq: r.t.}) that both $\mathbf{c}$ and $\mathbf{d}$
are degenerate clusters; if the latter holds, then we have by (\ref{eq: r.t.})
that both $\mathbf{c}$ and $\mathbf{d}$ are nondegenerate clusters
and $\left\vert \mathbf{c}\right\vert <\left\vert \mathbf{d}\right\vert $.
In either case, there is a function $f$ from $\mathbf{d}$ onto $\mathbf{c}$
that reduces $\mathfrak{G}$ to $\mathfrak{F}$. 

Consider $heit(\mathfrak{rt}(\mathfrak{F}))=k$. Let $\mathfrak{rt}(\mathfrak{F})=\left\langle \left(\mathbf{r},s\right),\left(\mathfrak{t}_{1},\ldots,\mathfrak{t}_{m}\right),\left(\mathfrak{t}_{m+1},\ldots,\mathfrak{t}_{m+n}\right)\right\rangle $
and $\mathfrak{rt}(\mathfrak{G})=\left\langle \left(\mathbf{r}',s'\right),\left(\mathfrak{t}'_{1},\ldots,\mathfrak{t}'_{k}\right),\left(\mathfrak{t}'_{k+1},\ldots,\mathfrak{t}'_{k+l}\right)\right\rangle $.
Since $\mathfrak{rt}(\mathfrak{F})\sqsubseteq\mathfrak{rt}(\mathfrak{\mathfrak{G}})$,
we have that 
\begin{enumerate}
\item $s\preccurlyeq s'$, 
\item $m=k$ and for each $1\leqslant i\leqslant m$, $\mathfrak{t}_{i}\sqsubseteq\mathfrak{t}'_{i}$,
\item either $l=n=0$, or $l\geqslant n>0$ and $\mathfrak{t}_{m+n}\sqsubseteq\mathfrak{t}'_{k+l}$
and there are $j_{m+1},\ldots,j_{m+n}$ such that $k+1\leqslant j_{m+1}<\cdots<j_{m+n}\leqslant k+l$,
and $\mathfrak{t}_{h}\sqsubseteq\mathfrak{t}'_{j_{h}}$ for each $h$
with $m+1\leqslant h\leqslant m+n$. 
\end{enumerate}
Apply the same reason as the base case, we have by (i) that there
is a function $f$ from $\mathbf{r}'$ onto $\mathbf{r}$ such that
$f$ reduces $\mathfrak{G}\upharpoonright\mathbf{r}'$ to $\mathfrak{F}\upharpoonright\mathbf{r}$.
Since $heit(\mathfrak{rt}(\mathfrak{F}))=k$, the heights of $\mathfrak{t}_{1},\ldots,\mathfrak{t}_{m},\mathfrak{t}_{m+1},\ldots,\mathfrak{t}_{m+n}$
are all less than $k$, and hence by (ii), (iii) and induction hypothesis,
we have that for each $1\leqslant i\leqslant m$, there is a function
$f_{i}$ that reduces $\mathfrak{G}\upharpoonright(\bigcup dom(\mathfrak{t}'_{i}))$
to $\mathfrak{F}\upharpoonright(\bigcup dom(\mathfrak{t}_{i}))$,
and for each $h$ with $m+1\leqslant h\leqslant m+n$, there is a
function $f_{h}$ that reduces $\mathfrak{G}\upharpoonright(\bigcup dom(\mathfrak{t}'_{j_{h}}))$
to $\mathfrak{F}\upharpoonright(\bigcup dom(\mathfrak{t}_{h}))$.
Let $\neg J=\{k+1,\ldots,k+l\}-\{j_{m+1},\ldots,j_{m+n}\}$. It follows
from (iii) that $\mathfrak{t}_{m+n}\sqsubseteq\mathfrak{t}'_{k+l}$,
and hence $\left|root(\mathfrak{t}_{m+n})\right|\preccurlyeq\left|root(\mathfrak{t}'_{k+l})\right|$.
We then have by (\ref{eq: tree representation}) that $0<\left|root(\mathfrak{t}_{m+n})\right|\leqslant min\{\left|root(\mathfrak{t}'_{j})\right|:j\in\neg J\}$,
and thus $root(\mathfrak{t}_{m+n})$ and elements of $\{root(\mathfrak{t}'_{j}):j\in\neg J\}$
are nondegenerate clusters. Let $g$ be any function from $\bigcup\bigcup_{j\in\neg J}dom(\mathfrak{t}'_{j})$
onto $\bigcup dom(\mathfrak{t}_{m+n})$ such that $g[root(\mathfrak{t}'_{j})]=root(\mathfrak{t}_{m+n})$
for each $j\in\neg J$. This is possible because of $\left|root(\mathfrak{t}_{m+n})\right|\leqslant min\{\left|root(\mathfrak{t}'_{j})\right|:j\in\neg J\}$.
It is easy to see that $g$ reduces $\mathfrak{G}\upharpoonright(\bigcup\bigcup_{j\in\neg J}dom(\mathfrak{t}'_{j}))$
to $\mathfrak{F}\upharpoonright(\bigcup dom(\mathfrak{t}_{m+n}))$.
Finally, let $h=f\cup g\cup\{f_{i}\}_{1\leqslant i\leqslant m+n}$.
It is routine to check that $h$ reduces $\mathfrak{G}$ to $\mathfrak{F}$.\end{proof} 

Recall that for each nonempty $X\subseteq W$, we use $\mathfrak{F}\upharpoonright X$
for the restriction of $\mathfrak{F}$ to $X$. Apply Proposition
\ref{prop: weak width frame condition}, the following fact is easily
verifiable.

\begin{fact}\label{fact: union of trees} Let $\mathfrak{F}=\left\langle W,R\right\rangle $
be any rooted finite transitive frame for $\mathsf{Wid}_{1}^{+}$,
and let $\mathbf{c}$ be the initial cluster in $\mathfrak{F}$. Then
there are disjoint subframes $\mathfrak{F}_{1},\ldots\mathfrak{F}_{n}$
of $\mathfrak{F}$ such that $\mathfrak{F}\upharpoonright(\left.\mathbf{c}\right\uparrow ^{-})=\biguplus_{1\leqslant i\leqslant n}\mathfrak{F}_{i}$
and $\mathfrak{sk}(\mathfrak{F}_{i})^{-1}$ is a finite tree for each
$1\leqslant i\leqslant n$. \end{fact} 

Let $\mathfrak{F}=\left\langle W,R\right\rangle $ be any rooted finite
transitive frame for $\mathsf{Wid}_{1}^{+}$ and let $\mathbf{c}$
be the initial cluster in $\mathfrak{F}$. According to Fact \ref{fact: union of trees},
there there are disjoint subframes $\mathfrak{F}_{1},\ldots\mathfrak{F}_{h}$
of $\mathfrak{F}$ such that $\mathfrak{F}\upharpoonright\left.\mathbf{c}\right\uparrow ^{-}=\biguplus_{1\leqslant i\leqslant h}\mathfrak{F}_{i}$
and $\mathfrak{sk}(\mathfrak{F}_{i})^{-1}$ is a finite tree for each
$1\leqslant i\leqslant h$. Let $\mathbf{T}=\{\mathfrak{rt}(\mathfrak{F}_{1}),\ldots,\mathfrak{rt}(\mathfrak{F}_{h})\}$,
and assume that $\{\mathfrak{t}_{1},\ldots,\mathfrak{t}_{m}\}=\{\mathfrak{t}\in\mathbf{T}:\mathfrak{t}=\left\langle \mathfrak{T},\tau\right\rangle \wedge\tau(root(\mathfrak{t}))=0\}$
and $\{\mathfrak{t}_{m+1},\ldots,\mathfrak{t}_{m+n}\}=\{\mathfrak{t}\in\mathbf{T}:\mathfrak{t}=\left\langle \mathfrak{T},\tau\right\rangle \wedge\tau(root(\mathfrak{t}))>0\}$
with $\tau_{m+n}(root(\mathfrak{t}_{m+n}))=min\{\tau_{i}(root(\mathfrak{t}_{i})):m\leqslant i\leqslant m+n\}$,
in which $\tau_{i}$ is the labeling function in $\mathfrak{t}_{i}$.
The\emph{ standard representation} \emph{tree} of $\mathfrak{F}$
is the following $\omega$-tree: 
\[
\mathfrak{srt}(\mathfrak{F})=\left\langle \left(\mathbf{c},\tau(\mathbf{c})\right),\left(\mathfrak{t}_{1},\ldots,\mathfrak{t}_{m}\right),\left(\mathfrak{t}_{m+1},\ldots,\mathfrak{t}_{m+n}\right)\right\rangle ,
\]
where for each $\mathbf{c}\in\mathfrak{sk}(W)$, $\tau(\mathbf{c})=\left\vert \mathbf{c}\right\vert $
if $\mathbf{c}$ is a nondegenerate cluster in $\mathfrak{F}$, otherwise
$\tau(\mathbf{c})=0$. Note that for any finite transitive frame $\mathfrak{F}$
that both $\mathfrak{srt}(\mathfrak{F})$ and $\mathfrak{rt}(\mathfrak{F})$
are well-defined, they are always different from each other, since
the root of $\mathfrak{srt}(\mathfrak{F})$ is the initial cluster
in $\mathfrak{F}$ and the root of $\mathfrak{rt}(\mathfrak{F})$
is the final cluster in $\mathfrak{F}$. Apply Lemma \ref{prop:tree-reducible-1},
the following Lemma can be proved in a similar way as the inductive
case in Lemma \ref{prop:tree-reducible-1}.

\begin{lemma} Let $\mathfrak{F}$ and $\mathfrak{G}$ be finite transitive
frames for $\mathsf{Wid}_{1}^{+}$, and let $\mathfrak{srt}(\mathfrak{F})\sqsubseteq\mathfrak{srt}(\mathfrak{\mathfrak{G}})$.
Then $\mathfrak{\mathfrak{G}}$ is reducible to $\mathfrak{F}$.\label{coro: covering function-1-1}\label{prop:srt-reducible-1-1}
\end{lemma}

\begin{lemma} Let $n,k\geqslant1$ and let $(\mathfrak{F}_{k})_{k\in\omega}$
be an infinite sequence of finite rooted transitive frames for $\mathsf{Wid}_{k}^{\bullet}$
of rank at most $m$ and of weak width $1$. Then there is an infinite
$I\subseteq\omega$ such that for all $i,j\in I$ with $i<j$, $\mathfrak{F}_{j}$
is reducible to $\mathfrak{F}_{i}$.\label{lem: (Fi) <=00003D k,n -> no odd seq}
\end{lemma}

\begin{proof} Since each $\mathfrak{F}_{i}$ is a frame for $\mathsf{Wid}_{k}^{\bullet}$
of rank at most $m$, we have by Proposition \ref{prop: frame cond 4 Wid*n}
that there are at most $m\times k$ degenerate clusters in $\mathfrak{F}_{i}$,
and hence $\mathfrak{srt}(\mathfrak{F}_{i})\in\mathbf{T}_{\leqslant m,<m\times k+1}^{\omega}$
for each $i\in\omega$. We then obtain by Theorem \ref{thm: wqo on trees}
that there is an infinite $I\subseteq\omega$ such that $(\mathfrak{srt}(\mathfrak{F}_{i}))_{i\in I}$
is an infinite $\sqsubseteq$-chain, and hence by Lemma \ref{prop:srt-reducible-1-1},
$\mathfrak{F}_{j}$ is reducible to $\mathfrak{F}_{i}$ for all $i,j\in I$
with $i<j$. \end{proof}

\begin{theorem} \label{thm:K4+Bn+W*k-FA}For all $n,k\geqslant1$,
all extensions of $\mathbf{K4B}_{n}\oplus\{\mathsf{Wid}_{1}^{+},\mathsf{Wid}_{k}^{\bullet}\}$
are finitely axiomatizable, and are hence decidable. \end{theorem}

\begin{proof} Let $\mathbf{L}=\mathbf{K4B}_{n}\oplus\{\mathsf{Wid}_{1}^{+},\mathsf{Wid}_{k}^{\bullet}\}$
with $n,k\geqslant1$. By Theorem \ref{thm: Segerberg, depth}, all
extensions of $\mathbf{L}$ have the f.m.p.\ To show that all extensions
of $\mathbf{L}$ are finitely axiomatizable, it then suffices by Theorem
\ref{coro: trans L fa iff} to let $\{\mathfrak{F}_{i}\}_{i\in\omega}$
be any infinite sequence of finite rooted frames for $\mathbf{L}$
and show that it is not irreducible. For each $i\in\omega$, because
$\mathfrak{F}_{i}$ is a frame for $\mathsf{B}_{n}$ and $\mathsf{Wid}_{1}^{+}$,
it is clear by Propositions \ref{prop: frame cond 4 B_k} and \ref{prop: weak width frame condition}
that $\mathfrak{F}_{i}$ is of rank at most $n$ and of weak width
$1$. Then by Lemma \ref{lem: (Fi) <=00003D k,n -> no odd seq}, $\mathfrak{F}_{j}$
is reducible to $\mathfrak{F}_{i}$ for some $i,j\in\omega$ with
$i<j$, and hence $\{\mathfrak{F}_{i}\}_{i\in\omega}$ is not irreducible.
\end{proof}

Since $\mathbf{S4B}_{n}$ is an extension of $\mathbf{K4B}_{n}\oplus\{\mathsf{Wid}_{k}^{\bullet}\}$
for all $n,k\geqslant1$, the following Corollary follows immediately
from Theorem \ref{thm:K4+Bn+W*k-FA}:

\begin{corollary} For all $n\geqslant1$, all extensions of $\mathbf{S4B}_{n}\oplus\{\mathsf{Wid}_{1}^{+}\}$
are finitely axiomatizable, and are hence decidable.\end{corollary}

\section{Conclusion\label{sect: remarks}}

In this paper, we proved as our negative result that there are non-finitely-axiomatizble
extensions of $\mathbf{K4B}_{n}\oplus\mathsf{Wid}_{k}^{+}$ for all
$n\geqslant3$ and $k\geqslant2$, by a way of constructing infinite
irreducible sequences of finite rooted transitive frames of depth
$3$ and of weak width $2$. As our positive result, we showed that
all extensions of $\mathbf{K4B}_{n}\oplus\{\mathsf{Wid}_{1}^{+},\mathsf{Wid}_{k}^{\bullet}\}$
are finitely axiomatizable for all $n,k\geqslant1$, by a way of applying
\emph{wqo} on finite height $\omega$-trees. It can be shown that
that there are non-finitely-axiomatizble extensions of $\mathbf{K4B}_{n}\oplus\mathsf{Wid}_{k}^{\bullet}$
for all $n\geqslant3$ and $k\geqslant1$. Therefore formulas $\mathsf{Wid}_{1}^{+}$
play an essential role in our finite axiomatizability result. However,
the following problem still remains open: for each $n\geqslant1$,
are all extensions of $\mathbf{K4B}_{n}\oplus\mathsf{Wid}_{1}^{+}$
finitely axiomatizable? Finally, since the infinite irreducible sequences
of frames constructed in section \ref{sec:Transitive-Logics-of FD and FSW}
don't validate any formula $\mathsf{Wid}_{k}^{\bullet}$. So the following
problem is unsettled: for each $n,k\geqslant1$ and $m\geqslant2$,
are all extensions of $\mathbf{K4B}_{n}\oplus\{\mathsf{Wid}_{m}^{+},\mathsf{Wid}_{k}^{\bullet}\}$
finitely axiomatizable?

\providecommand{\bysame}{\leavevmode\hbox to3em{\hrulefill}\thinspace}
\providecommand{\MR}{\relax\ifhmode\unskip\space\fi MR }
\providecommand{\MRhref}[2]{%
  \href{http://www.ams.org/mathscinet-getitem?mr=#1}{#2}
}
\providecommand{\href}[2]{#2}

\end{document}